\documentclass[11pt,oneside]{amsart}
\usepackage{bm, calc,geometry, verbatim, graphicx, amssymb, color, mathabx, mathtools, enumitem, bm, mathrsfs, amsmath, amsthm, ulem, xspace,tabularx }
\usepackage[usenames,dvipsnames,svgnames,table]{xcolor}
\usepackage[pdftex,bookmarks,colorlinks,breaklinks]{hyperref} 

\hypersetup{linkcolor=blue,citecolor=red,filecolor=dullmagenta,urlcolor=darkblue} \newcommand\brho{\operatorname{\boldsymbol{\rho}}}

\newtheorem{theorem}{Theorem}[section]
\newtheorem{corollary}[theorem]{Corollary}

\newtheorem{remark}[theorem]{Remark}

\newcommand\dg{\operatorname{\textup{{\fontfamily{ptm}\selectfont deg}}}}
\newcommand\edg{\operatorname{\textup{{\fontfamily{ptm}\selectfont e}}}}

\usepackage{physics}

\newcommand{\longeq}{\scalebox{3}[1]{=}}

      \makeatletter 
      \def\@setcopyright{}
      \def\serieslogo@{}
      \makeatother      

\begin{document}
   \author{Amin  Bahmanian}
   \address{Department of Mathematics,
  Illinois State University, Normal, IL USA 61790-4520}

\title[On Factors with Prescribed Degrees in Bipartite Graphs] {On Factors with Prescribed Degrees in Bipartite Graphs}

\begin{abstract} 
We establish a new criterion for a bigraph to have a subgraph with prescribed degree conditions. We show that the bigraph $G[X,Y]$ has a spanning subgraph $F$ such that $g(x)\leq \dg_F(x) \leq f(x)$ for $x\in X$ and  $\dg_F(y) \leq f(y)$ for $y\in Y$ if and only if  $\sum\nolimits_{b\in B} f(b)\geq  \sum\nolimits_{a\in A} \max \big\{0, g(a) - \dg_{G-B}(a)\big\}$ for  $A\subseteq X, B\subseteq Y$. Using Folkman-Fulkerson's Theorem, Cymer and Kano found a different criterion for the existence of such a subgraph  (Graphs Combin. 32 (2016), 2315--2322). Our proof is self-contained and relies  on alternating path technique.  As an application, we prove the following extension of Hall's theorem. A bigraph $G[X,Y]$ in which each edge has multiplcity at least $m$ has a subgraph $F$ with $g(x)\leq \dg_F(x)\leq f(x)\leq \dg(x)$ for $x\in X$,  $\dg_F(y)\leq m$ for $y\in Y$ if and only if
$\sum_{y\in N_G(S)}f(y)\geq \sum_{x\in S}g(x)$ for  $S\subseteq X$. 
\end{abstract}

   \date{\today}
   \maketitle
\section{Introduction}  
Factor theory is one of the oldest and most active areas of graph theory \cite{MR2816613}, that  started in the 19th century when Petersen showed that every even regular graph is 2-factorable.  In this note, we are primarily concerned with factors with prescribed degree conditions in bigraphs.

A bigraph $G$ with bipartition $\{X, Y \}$ will be denoted by $G[X, Y ]$, and for $S\subseteq X$, $\overline{S}$ means $X\backslash S$. For a real-valued function $f$ on a domain $D$ and  $A\subseteq D$,  $f(A) :=\sum\nolimits_{a\in A}f(a)$. For a graph $G=(V,E)$,   $u\in V$  and $A\subseteq V$, $\dg_G(u)$, and $\edg_G(uA)$ denote the number of edges incident with $u$,  and the number of edges between $u$ and $A$,  respectively. Let $f,g$ be  integer functions on the vertex set of a graph $G$ such that $0\leq g(x)\leq f(x)$ for all $x$. A {\it $(g,f)$-factor} is a spanning subgraph $F$ of $G$ with the property that $g(x)\leq \dg_F(x)\leq f(x)$ for each $x$.   An {\it $f$-factor} is an $(f,f)$-factor. Ore \cite{MR77920, MR83725} showed that  $G[X,Y]$  has an $f$-factor  if and only if $f(X)=f(Y)$ and
$$f(A)\leq \sum\nolimits_{y\in Y} \min\big\{f(y), \edg_G(y A)\big\} \quad \forall A\subseteq X.$$
Folkman and Fulkerson proved a $(g,f)$-factor theorem for bigraphs \cite{MR268065} which was simplified by Heinrich et al.  (Here, $x\dotdiv y$ means $\max\{0,x-y\}$). 
\begin{theorem} \cite[Theorem 1]{MR1081839}
The bigraph $G[X,Y]$ has a  $(g,f)$-factor if and only if
\begin{align*}
f(A)&\geq  \sum\nolimits_{u\notin A} \Big( g(u) \dotdiv \dg_{G-A}(u) \Big) \quad \forall A\subseteq X\cup Y.
\end{align*}
\end{theorem}
Recently  Cymer and Kano found another simple criteria.
\begin{theorem} \cite[Theorem 5]{MR3564794} \label{Cymer-Kanogfthm}
The bigraph $G[X,Y]$ has a  $(g,f)$-factor if and only if  the following conditions hold.
\begin{align*}
     g(A)&\leq \sum\nolimits_{y\in Y} \min \Big\{ f(y), \edg_G(yA)\Big\}
\quad &\forall A\subseteq X,\\
 g(B)&\leq \sum\nolimits_{x\in X} \min \Big\{ f(x), \edg_G(xB)\Big\}
\quad &\forall B\subseteq Y.\\
\end{align*}
\end{theorem}
Theorem \ref{Cymer-Kanogfthm} has been particularly useful in solving various generalized Sudoku puzzles \cite{RyserrhoBah, RyserrhosymBah}; Solving some of these puzzles can be reduced to finding $(g,f)$-factors with the additional property that  $g(y)=0$ for $y\in Y$ in a bigraph $G[X,Y]$. Motivated by solving such  problems, we establish the following new criterion for a bigraph to have a factor with prescribed degrees. 
\begin{theorem} \label{hallgen}
A bigraph $G[X,Y]$ has a  $(g,f)$-factor  with $g(y)=0$ for $y\in Y$  if and only if
\begin{equation} \label{refinedgfthm}
    f(B)\geq  \sum\nolimits_{x\in A} \Big( g(x) \dotdiv \edg_G(x \overline{B}) \Big) \quad \forall A\subseteq X, B\subseteq Y.
\end{equation}
\end{theorem}
While  Theorem \ref{Cymer-Kanogfthm} relies on Folkman-Fulkerson's $(g,f)$-factor theorem, 
our proof is self-contained and relies on alternating path technique \cite{MR1081839}. Before we prove our main result, we provide the following corollary. Here, $N_G(S)$ is the neighborhood of $S$ in $G$. 
\begin{corollary} \label{hallgencor}
A bigraph $G[X,Y]$ in which the mutiplicity of each edge is at least $m$, has a  $(g,f)$-factor  with $f(y)\leq m, g(y)=0$ for $y\in Y$  if and only if
\begin{equation} \label{refinedgfthmcor}
    f(N_G(S))\geq  g(S) \quad \forall S\subseteq X.
\end{equation}
\end{corollary}
\begin{proof}
By Theorem \ref{hallgen},  $G[X,Y]$ has a  $(g,f)$-factor  with $g(y)=0$ for $y\in Y$  if and only if 
\begin{equation} \label{refinedgfthmGMT}
   f(B)\geq  \sum\nolimits_{x\in A} \Big( g(x) \dotdiv \edg_G(x \overline B) \Big) \quad \forall A\subseteq X, B\subseteq Y.
\end{equation}
To complete the proof, we show that \eqref{refinedgfthmcor} and \eqref{refinedgfthmGMT} are equivalent. First, let us assume that \eqref{refinedgfthmGMT} holds, and let $A=S\subseteq X, B=N_G(S)\subseteq Y$. We have
\begin{equation} 
    f(N_G(S))\geq  \sum\nolimits_{x\in S} \Big( g(x) \dotdiv \edg_G(x \overline {N_G(S)}) \Big)=\sum\nolimits_{x\in S}  g(x)=g(S),
\end{equation}
and so \eqref{refinedgfthmcor} is satisfied. Conversely, assume that \eqref{refinedgfthmcor} holds, and let $A\subseteq X, B\subseteq Y$. Let
$$
S=\{x\in A \ |\ g(x)\geq  \edg_G(x \overline B) \}. 
$$
If we  show that $g(S)\leq f(B)+\edg_G(S \overline B)$, then we are done. 
We have
$$
g(S)\leq f(N_G(S))=f(N_G(S)\cap B)+ f(N_G(S)\backslash \overline B)\leq f(B) + m |N_G(S)\backslash B| \leq f(B) + \edg_G(S \overline B).
$$
\end{proof}
\begin{remark} \textup{
The case $m=1$ of Corollary \ref{hallgencor} was previously settled in \cite[Theorem 7]{MR3564794}. Observe that the case $m=1$, $f(x)=g(x)=1$ for $x\in X$ of Corollary \ref{hallgencor}  corresponds to the famous Hall's marriage theorem. 
}\end{remark}
\section{Proof of Theorem \ref{hallgen}}
To prove the necessity, suppose that $G$ has a $(g,f)$-factor $F$, and let $A\subseteq X, B\subseteq Y$. Define $C=\{x\in A \ |\ g(x) > \edg_G(x\overline{B})\}$. If $C=\emptyset$, then \eqref{refinedgfthm} is trivial. Otherwise, let $x\in C$. There must be at least $g(x) - \edg_G(x\overline{B})$ edges in $F$ joining $x$ to vertices in $B$. Hence, 
\begin{equation} 
\sum\nolimits_{x\in C} \Big( g(x) - \edg_G(x\overline{B}) \Big) \leq \edg_F(C B) \leq  \edg_F(A B) \leq  \sum\nolimits_{y\in B}\dg_F(y) \leq   f(B).  
\end{equation}
To prove the sufficiency, suppose that \eqref{refinedgfthm} holds. Let $F$ be a $(0,f)$-factor that minimizes $\delta:=\sum\nolimits_{x\in X} \Big( g(x) \dotdiv \dg_{F}(x) \Big)$. If $\delta=0$, then $F$ is a $(g,f)$-factor and we are done. So let us assume that $\delta  > 0$, and so 
\begin{equation} \label{Rdef}
R:=\{x\in X \ |\ g(x)>\dg_F(x)\}\neq \emptyset.
\end{equation}
To complete the proof, we find sets $A\subseteq X, B\subseteq Y$ such that \eqref{refinedgfthm} fails. A path (possibly of length zero) is {\it nice} if it starts with a vertex in $R$ and an edge in $E(G)\backslash E(F)$ and whose edges are alternately in $G-F$ and $F$. 
Let $W$ be the  set  of terminal vertices of all nice paths. Let $A=R\cup S$ where  $S:=(X\backslash R)\cap W$, and let  $B=Y\cap W$. We claim that
\begin{enumerate}
    \item [(a)] If $e\in E(F)$ with $e=xy$ and $y\in B$, then $x\in A$. 
    \item [(b)] If $e\in E(G)\backslash E(F)$ with $e=xy$ and $x\in A$, then $y\in B$.
    \item [(c)] $\edg_G(x\overline{B})=\dg_F(x)-e_G(xB)$ for $x\in A$.
    \item [(d)] $\dg_F(y)=f(y)$ for $y\in B$.
    \item [(e)] $\dg_F(x)=g(x)$ for $x\in S$.
    \item [(f)] $\dg_{G-B}(x)\leq  g(x)$ for $x\in A$.
\end{enumerate}
Observe that (c) is an immediate consequence of (b), and \eqref{Rdef} and (c) imply (f). To prove (a) and (b), let $e=xy$. If $y\in B$, there is  a nice path $P$ ending at $y$ (whose last edge is not in $F$), and so if $e\in E(F)$ and $x\notin R$, then $P+ex$ is a nice path ending at $x$, and consequently,  $x\in S$. Similarly, if $x\in A$,  there is  a nice path $P$ ending at $x$ (whose last edge is  in $F$), and so if $e\in E(G)\backslash E(F)$, then $P+ey$ is a nice path ending at $y$, and consequently,  $y\in B$. 
To prove (d), let $y\in B$. There is a nice path $P$ ending at $y$. If $\dg_F(y)<f(y)$, then since the last edge of $P$ is in $E(G)\backslash E(F)$, the $(0,f)$-factor $F'$  with $E(F')=E(F)\Delta E(P)$ contradicts the minimality of $\delta$ (We use $\Delta$ for the symmetric difference). Similarly, to prove (e), let $x\in S$. There is a nice path $P$ ending at $x$. If $\dg_F(x)>g(x)$, then  since the last edge of $P$ is in $E(F)$, the $(0,f)$-factor $F'$  with $E(F')=E(F)\Delta E(P)$ contradicts the minimality of $\delta$. The following completes the proof.
 \begin{align*}
 \sum\nolimits_{x\in A} \Big( g(x) \dotdiv \edg_G(x\overline{B}) \Big)\mathop {\longeq}  \limits^{\text{(f)}}&\sum\nolimits_{x\in A} \Big( g(x) - \edg_G(x\overline{B}) \Big)\\
 \mathop {\longeq} \limits^{\text{(c)}}&\sum\nolimits_{x\in A} \Big( g(x) - \dg_F(x)+\edg_G(xB) \Big)\\
  \mathop {\longeq} \limits^{\text{(e)}} & \sum\nolimits_{x\in R} \Big( g(x) - \dg_{F}(x) \Big) +  \edg_G(AB)\\
    \mathop  {\quad >\quad }\limits^{\eqref{Rdef}}  &  \edg_G(AB)\\
    \mathop {\longeq} \limits^{\text{(a)}}&  \sum\nolimits_{y\in B}\dg_F(y)\\
    \mathop {\longeq} \limits^{\text{(d)}}& f(B).
\end{align*}
This completes the proof.


\begin{thebibliography}{1}

\bibitem{MR2816613}
Jin Akiyama and Mikio Kano.
\newblock {\em Factors and factorizations of graphs}, volume 2031 of {\em
  Lecture Notes in Mathematics}.
\newblock Springer, Heidelberg, 2011.
\newblock Proof techniques in factor theory.

\bibitem{RyserrhoBah}
Amin Bahmanian.
\newblock Ryser's theorem for $\brho$-latin rectangles.
\newblock {\em arxiv, Submitted for Publication}.

\bibitem{RyserrhosymBah}
Amin Bahmanian.
\newblock Ryser's theorem for symmetric $\brho$-latin squares.
\newblock {\em Submitted for Publication}.

\bibitem{MR3564794}
R.~Cymer and Mikio Kano.
\newblock Generalizations of marriage theorem for degree factors.
\newblock {\em Graphs Combin.}, 32(6):2315--2322, 2016.

\bibitem{MR268065}
Jon Folkman and D.~R. Fulkerson.
\newblock Flows in infinite graphs.
\newblock {\em J. Combinatorial Theory}, 8:30--44, 1970.

\bibitem{MR1081839}
Katherine Heinrich, Pavol Hell, David~G. Kirkpatrick, and Gui~Zhen Liu.
\newblock A simple existence criterion for {$(g<f)$}-factors.
\newblock {\em Discrete Math.}, 85(3):313--317, 1990.

\bibitem{MR77920}
Oystein Ore.
\newblock Studies on directed graphs. {I}.
\newblock {\em Ann. of Math. (2)}, 63:383--406, 1956.

\bibitem{MR83725}
Oystein Ore.
\newblock Graphs and subgraphs.
\newblock {\em Trans. Amer. Math. Soc.}, 84:109--136, 1957.

\end{thebibliography}
\end{document}